\documentclass[12pt]{amsart}

\usepackage{amssymb,amsmath,amsthm}

\newtheorem{theorem}{Theorem}[section]
\newtheorem{lemma}[theorem]{Lemma}
\newtheorem{prop}[theorem]{Proposition}
\newtheorem{cor}[theorem]{Corollary}
\newtheorem{conjecture}[theorem]{Conjecture}
\newtheorem{defn}{Definition}

\def \R{\mathbb{R}}

\def \e{\varepsilon}
\def \ph{\varphi}
\def \Om{\Omega}

\def \grad{\nabla}

\def \half{\frac{1}{2}}

\numberwithin{equation}{section}

\title{Inverse Radiative Transport with Local Data}

\author[Chung]{Francis J. Chung}
\address{Department of Mathematics, University of Kentucky, Lexington KY 40506, USA}
\email{fjchung@gmail.com}

\thanks{The author is supported in part by Simons Collaboration Grant 582020. This paper was inspired in part by the author's attendance at the KI-Net conference on forward and inverse kinetic theory organized by Qin Li at the University of Wisconsin Madison in October 2019. The author would also like to thank Plamen Stefanov for a helpful conversation.}

\begin{document}

\begin{abstract}
We consider an inverse problem for a radiative transport equation (RTE) in which boundary sources and measurements are restricted to a single subset $E$ of the boundary of the domain $\Om$.  We show that this problem can be solved globally if the restriction of the X-ray transform to lines through $E$ is invertible on $\Om$. In particular, if $\Om$ is strictly convex, we show that this local data problem can be solved globally whenever $E$ is an open subset of the boundary.  The proof relies on isolation and analysis of the second term in the collision expansion for solutions to the RTE, essentially considering light which scatters exactly once inside the domain.  
\end{abstract}

\maketitle

\section{Introduction}

The analysis of optical imaging in scattering media is a problem of great practical interest \cite{ArrSch}. This problem is usually analyzed under the assumption that optical data is available on the entire boundary of the domain being imaged. In many applications, however, it is often impossible or impractical to place sources and obtain measurements on the entire boundary. In this paper, we consider an optical imaging problem in a scattering domain, in which both the sources of light and the measurements are restricted to a single subset of the boundary. 

We analyze this problem in a scattering medium where light propagation can be modeled by a radiative transport equation (RTE).  To simplify things, we consider an equation of the type
\begin{equation}\label{RTE}
\theta \cdot \grad_x u(x,\theta) = \sigma(x) \left( \int_{S^{n-1}} u(x,\theta) d\theta - u(x,\theta) \right),
\end{equation}
on a smooth bounded domain $\Om \subset \R^n$, $n \geq 3$, with a single positive coefficient $\sigma \in C(\Om)$. Here $u : \Om \times S^{n-1} \rightarrow \R$ is the specific intensity of light, so  $u(x,\theta)$ represents the intensity of light at $x$ in direction $\theta$.  Equation \eqref{RTE} can be thought of as modeling light propagation in a domain dominated by isotropic scattering, with negligible absorption.  

If we adopt the notation
\[
\langle u \rangle (x) = \frac{1}{|S^{n-1}|}\int_{S^{n-1}} u(x,\theta) d\theta
\]
for the angular average, we can write our RTE \eqref{RTE} as
\begin{equation}\label{ShortRTE}
\theta \cdot \grad u = \sigma (\langle u \rangle - u).
\end{equation}

We define $\partial \Om_{\pm}$ by 
\[
\partial \Om_{\pm} = \{ (x,\theta) | x \in \partial \Om, \pm \nu(x) \cdot \theta > 0 \}
\]
where $\nu(x)$ is the outward normal at $x$. The following result guarantees existence and uniqueness of solutions to \eqref{ShortRTE}.

\begin{prop}\label{ExistUnique}
Suppose $f \in L^p(\partial \Om_-)$, with $1 \leq p \leq \infty$. Then there exists a unique $u \in L^p(\Om \times S^{n-1})$ such that
\begin{equation}\label{FullRTE}
\begin{split}
&\theta \cdot \grad u = \sigma (\langle u \rangle - u) \\
& u|_{\partial \Om_-} = f.  \\
\end{split}
\end{equation}
Moreover there exist constants $C_1, C_2$ depending only on $\Om$ and $\sigma$, and a trace operator defining $u|_{\partial \Om_+}$ such that 
\[
\|u\|_{L^p(\partial \Om_+)} \leq C_1\|u\|_{L^p(\Om \times S^{n-1})} \leq C_2\|f\|_{L^p(\partial \Om_-)}.
\]
\end{prop}

For a proof, see \cite{EggSch} (also Proposition \ref{CollisionExpansion} below).  Proposition \ref{ExistUnique} allows us to define the albedo operator 
\[
\mathcal{A}: L^1(\partial \Om_-) \rightarrow L^1(\partial \Om_+)
\]
by
\[
\mathcal{A}(f) = u|_{\partial \Om_+}.
\]
The standard problem of optical tomography in this setting, solved in \cite{ChoSte}, is to determine $\sigma$ from knowledge of the full map $\mathcal{A}$. 

To define our \emph{local data} problem, in which only a subset of the boundary of $\Omega$ is accessible for placing sources and taking measurements, suppose $E \subset \partial \Om$, and define
\[
E_{\pm} = \{ (x,\theta) | x \in E, \pm \nu(x) \cdot \theta \geq 0 \}.
\]
We define a new map
\[
\mathcal{A}|_E: L^1(E_-) \rightarrow L^1(E_+)
\]
by 
\[
\mathcal{A}|_E(f) = u|_{E_+}
\]
for $f \in L^1(\partial \Om)$ supported in $E_-$.  

We want to know if knowledge of $\mathcal{A}|_E$ determines $\sigma$.  To state our result, we make the following definition regarding the invertibility of the restriction of the X-ray transform to $E$.

\begin{defn}\label{Scannability}
Suppose that $E \subset \partial \Om$. We say that $\Om$ can be \emph{scanned} from $E$ if for $\ph \in C(\Om)$, the condition that 
\[
\int_{\ell} \ph \,d \ell = 0 \mbox{ for each line segment } \ell \subset \bar{\Om} \mbox{ with } E \cap \ell \neq \varnothing
\]
implies that $\ph \equiv 0$.
\end{defn}

The main result of this paper is the following.

\begin{theorem}\label{MainThm}
Suppose $E \subset \partial \Om$, and $\Om$ can be scanned from $E$. Then $\mathcal{A}|_E$ determines $\sigma$ in the sense that if $\sigma_1, \sigma_2 \in C(\Om)$ are positive functions producing corresponding maps $\mathcal{A}_1|_E, \mathcal{A}_2|_E$, then 
\[
\mathcal{A}_1|_E = \mathcal{A}_2|_E \mbox{ implies } \sigma_1 = \sigma_2.
\]
\end{theorem}

The following corollary gives a stronger but more easily checked condition on $E$ for $\mathcal{A}|_E$ to determine $\sigma$, showing that Theorem \ref{MainThm} applies to a wide range of cases. 

\begin{cor}\label{ConvexCor}
Suppose $\Om$ is strictly convex, and $E$ is an open subset of $\partial \Om$.  Then $\mathcal{A}|_E$ determines $\sigma$.
\end{cor}

It is natural to ask if the result in Theorem \ref{MainThm} can be generalized to remove the scannability condition.

\begin{conjecture}\label{LocalConjecture}
Suppose $E \subset \partial \Om$ is open.  Then $\mathcal{A}|_E$ determines $\sigma$.
\end{conjecture}

A similar conjecture in the highly-scattering limit, where light propagation can be modeled by an elliptic equation, has been the subject of much work; see \cite{KenSal} for a survey.  However, this author is not aware of similar efforts in the transport case. 

The remainder of this paper is organized as follows.  Notation and preliminary results are presented in Section 2. 
The proofs of Theorem \ref{MainThm} and Corollary \ref{ConvexCor} are given in Section 3.  Finally, Section 4 is devoted to comments on Conjecture \ref{LocalConjecture}.  

\section{Preliminaries}

We want to analyze solutions to the RTE in terms of the collision expansion. To describe this expansion, we introduce the following notation.

If $x, y \in \bar{\Om}$, such that the line segment joining $x$ and $y$ is contained in $\Om$, we define the attenuation factor
\begin{equation}\label{OpticalDistance}
\alpha(x,y) = \exp\left(- \int_0^{\|y-x\|}\sigma(x + s(\widehat{y-x})) \, ds\right),
\end{equation}
where $\widehat{y-x}$ is the unit vector in the direction of $y-x$. We note that $\alpha$ is symmetric in $x$ and $y$, and for $t \geq 0$ 
\begin{equation}\label{OpticalDistanceDerivative}
\partial_t \alpha(x, x+ t\theta) = -\sigma(x + t\theta)\alpha(x, x+ t\theta).
\end{equation}
Roughly speaking $\alpha(x,y)$ represents the attenuation experienced by light traveling from $x$ to $y$ in the absence of the $\langle u \rangle$ term in the RTE. This notion can be made precise in the following two senses:

First, it follows from \eqref{OpticalDistanceDerivative} that the solution to the transport equation
\begin{equation}\label{JEquation}
\begin{split}
\theta \cdot \grad u +\sigma u &= 0 \\
u|_{\partial \Om_-} &= f \\
\end{split}
\end{equation}
is given by 
\begin{equation}\label{JDefn}
u(x,\theta) = Jf(x,\theta) := \alpha(x, x_{\theta-})f(x_{\theta-}, \theta),
\end{equation}
where $x_{\theta \pm}$ is the first point on $\partial \Om$ on the ray from $x$ in direction $ \pm \theta$. 

Second, the solution to the transport equation
\begin{equation}\label{TEquation}
\begin{split}
\theta \cdot \grad u + \sigma u &= S(x,\theta)\\
u|_{\partial \Om_-} &= 0 \\
\end{split}
\end{equation}
is given by 
\begin{equation}\label{TInverseDefn}
u(x,\theta) = T^{-1}S(x,\theta) := \int_0^{\|x - x_{\theta-}\|} \alpha(x, x - t\theta)S(x - t\theta, \theta)dt.
\end{equation}
Note that since $\alpha < 1$ and $\|x - x_{\theta-}\|$ is uniformly bounded on $\Om$, 
\begin{equation}\label{TInverseBound}
\|T^{-1} S\|_{L^{\infty}(\Om \times S^{n-1})} \lesssim \|S\|_{L^{\infty}(\Om \times S^{n-1})}.
\end{equation}

With this notation we have the following proposition.

\begin{prop}\label{CollisionExpansion}
Define the operator
\[
K(u) = T^{-1}\sigma\langle u \rangle.
\]
Then there exists $0 < C < 1$ such that
\begin{equation}\label{KBound}
\|K\|_{L^{\infty}(X \times {S^{n-1}}) \rightarrow L^{\infty}(X \times {S^{n-1}})} < C, 
\end{equation}
and if $u$ solves the RTE
\begin{equation}\label{BoundaryRTE}
\begin{split}
&\theta \cdot \grad u = \sigma (\langle u \rangle - u) \\
& u|_{\partial \Om_-} = f,  \\
\end{split}
\end{equation}
for some $f \in L^{\infty}(\partial \Om_{-})$, then $u$ takes the form
\begin{equation}\label{CollisionExp}
u = (1 + K + K^2 + \ldots)Jf.
\end{equation}
\end{prop}

We refer to the series \eqref{CollisionExp} as the collision expansion of $u$. The proof is essentially the proof of Theorem 1.1 in \cite{EggSch}, but with a more explicit series form for $u$ (compare e.g. to the existence results in \cite{ChoSte, ChuSch, DauLio}, although note that those have slightly different hypotheses). We include the full proof here for completeness, using our notation. 

\begin{proof}
We can rewrite \eqref{BoundaryRTE} as 
\begin{equation*}
\begin{split}
&(\theta \cdot \grad + \sigma)u = \sigma \langle u \rangle \\
& u|_{\partial \Om_-} = f.  \\
\end{split}
\end{equation*}
Since $Jf$ solves \eqref{JEquation}, we have
\begin{equation*}
\begin{split}
&(\theta \cdot \grad + \sigma)(u - Jf) = \sigma \langle u \rangle \\
& (u - Jf)|_{\partial \Om_-} = 0.  \\
\end{split}
\end{equation*}
Therefore $u - Jf$ solves \eqref{TEquation} with $S = \sigma \langle u \rangle$, and so
\[
u - Jf = T^{-1}\sigma \langle u \rangle.
\]
Rearranging, we have
\[
(I - K)u = Jf.
\]
Equation \eqref{KBound} implies that the Neumann series for $(I-K)^{-1}$ converges, which gives \eqref{CollisionExp}.

To prove \eqref{KBound}, note that explicitly
\[
Ku(x,\theta) = \int_0^{\|x - x_{\theta-}\|} \alpha(x, x - t\theta)\sigma(x - t\theta)\langle u\rangle (x - t\theta) dt .
\]
Since $\sigma$ and $\alpha$ are nonnegative, we have
\[
|Ku(x,\theta)| \leq \|u\|_{L^{\infty}(\Om \times S^{n-1})}\int_0^{\|x - x_{\theta-}\|} \alpha(x, x - t\theta)\sigma(x - t\theta) dt
\]
Equation \eqref{OpticalDistanceDerivative} implies that 
\[
\alpha(x, x - t\theta)\sigma(x - t\theta) = -\partial_t \alpha(x, x - t\theta),
\]
so 
\[
|Ku(x,\theta)| \leq \|u\|_{L^{\infty}(\Om \times S^{n-1})}(1 - \alpha(x, x_{\theta-})).
\]
Since $\sigma$ is uniformly positive, $(1 - \alpha(x, x_{\theta-}))$ is uniformly less than $1$, so
\[
|Ku(x,\theta)| \leq C\|u\|_{L^{\infty}(\Om \times S^{n-1})}
\]
with $0<C<1$ as desired. This finishes the proof of \eqref{KBound} and therefore the proposition.
\end{proof}

In the spirit of \cite{ChoSte}, we intend to focus (no pun intended) on highly singular boundary conditions for the RTE, concentrated at a point $(x_0,\theta_0) \in E_-$. We follow the approach of \cite{ChuSch} in using approximations to identity, and define $\psi_h: \partial \Om_- \rightarrow \R$ to be the product 
\[
\psi_h(x,\theta) = X_h(x)\Theta_h(\theta)
\]
where $X_h \in C(\partial \Om)$ and $\Theta_h \in C(S^{n-1})$ are nonnegative, supported only for $|x-x_0| \leq h$ and $|\theta - \theta_0| \leq h$, and normalized so $\|X_h\|_{C(\partial \Om)}, \|\Theta_h\|_{C(S^{n-1})} = O(h^{n-1})$ and
\[
\int_{\partial \Om} X_h = \int_{S^{n-1}} \Theta_h  = 1.
\]
Note that this implies that $X_h$ and $\Theta_h$ are approximations of identity in the sense that 
\begin{equation}\label{SeparatedApproximations}
\int_{\partial \Om} X_h \ph = \ph(x_0,\theta) + o_h(1) \mbox{ and } \int_{S^{n-1}} \Theta_h \ph = \ph(x,\theta_0) + o_h(1)
\end{equation}
for each $\ph \in C(\partial \Om \times S^{n-1})$ bounded uniformly in $h$. It follows that 
\[
\int_{\partial \Om \times S^{n-1}} \psi_h \ph = \ph(x_0,\theta_0) + o_h(1).
\]
For sufficiently small $h$, $\psi_h$ has support contained in $E_-$, so we can think of $\psi_h$ as a function in $L^p(E_-)$.  With the definition of this boundary condition, we are ready to tackle the proof of Theorem \ref{MainThm}.  

\section{Proofs of the Main Results}

We want to consider the solution to the equation
\begin{equation}\label{SingularSolution}
\begin{split}
&\theta \cdot \grad u = \sigma (\langle u \rangle - u) \\
& u|_{\partial \Om_-} = \psi_h  \\
\end{split}
\end{equation}

Since $K$ is an integrating operator, integrating over one spatial dimension and all $n-1$ dimensions of angular variables, later terms in the collision expansion \eqref{CollisionExp} tend to be smoother than earlier ones.  In particular, we have the following lemma for the tail of \eqref{CollisionExp} in the solution to \eqref{SingularSolution}.

\begin{lemma}\label{SmallRemainder}
Suppose $u$ solves \eqref{SingularSolution}. Then
\[
u = J\psi_h + KJ\psi_h + R
\]
where 
\[
\|R\|_{L^{\infty}(\Om \times S^{n-1})} = O(h^{2-n}).
\]
\end{lemma}

As we shall see below, the terms $J\psi_h$ and $K\psi_h$ can be $\sim h^{2-2n}$ and $\sim h^{1-n}$ respectively, so for small $h$, $R$ is genuinely nonsingular compared to the first two terms.  

\begin{proof}
By Proposition \ref{CollisionExpansion}, $R$ takes the form 
\begin{equation}\label{RExplicit}
R = K^2J\psi_h + K^3 J\psi_h + \ldots
\end{equation}
Therefore \eqref{KBound} implies that it suffices to show that 
\[
\|K^2J\psi_h\|_{L^{\infty}(\Om \times S^{n-1})} = O(h^{2-n}).
\]

To show this, we begin by establishing some estimates for $J\psi_h$ and $KJ\psi_h$. Explicitly,
\[
J\psi_h(x,\theta) = \alpha(x, x_{\theta-})\psi_h(x_{\theta-}, \theta) = \alpha(x, x_{\theta-})X_h(x_{\theta-})\Theta_h(\theta).
\]
$\Theta_h$ is only supported for $|\theta - \theta_0| \leq h$ and has total integral 1, so averaging over angles gives us
\[
\langle J\psi_h \rangle(x) \leq \sup_{|\theta - \theta_0|\leq h} |\alpha(x, x_{\theta-})X_h(x_{\theta-})| = O(h^{1-n}).
\]
Now $X_h$ is supported only for $|x - x_0| < h$, so $X_h(x_{\theta-})$ is nonzero only for $x$ which can be reached from this neighbourhood by traveling in direction $\theta$. If $|\theta - \theta_0| \leq h$, this means that $X_h(x_{\theta-})$ and $J\psi_h(x)$ are supported only for $x$  within $O(h)$ distance of the line $\{x_0 + t\theta_0\}$.  Then
\[
KJ\psi_h(x,\theta) = T^{-1}\sigma \langle J\psi_h \rangle,
\]
where the integral in the operator $T^{-1}$ is taken over the line through $x$ in direction $\theta$.  For $|\theta - \theta_0| \gg h$, the intersection of this line with the support of $J\psi_h$ is at most length $O(h)$, so 
\[
KJ\psi_h(x,\theta) \leq O(h)\|\langle J\psi_h \rangle\|_{L^{\infty}(\Om \times S^{n-1})} \leq O(h^{2-n})
\]
for $|\theta - \theta_0| \gg h$, and $KJ\psi_h(x,\theta) \leq O(h^{1-n})$ otherwise.  Integrating over angles now gives
\[
\langle KJ\psi_h(x,\theta) \rangle \leq O(h^{2-n}).
\]
The estimate \eqref{TInverseBound} gives us 
\[
K^2 J\psi_h(x,\theta) \leq O(h^{2-n}),
\]
which finishes the proof.

\end{proof}

\begin{proof}[Proof of Theorem \ref{MainThm}]

Fix $(x, \theta) \in E_-$. Define $\psi_h$ as above with $(x_0,\theta_0) = (x,\theta)$, and suppose $u$ solves \eqref{SingularSolution}.
Note that $(x,-\theta) \in E_+$, so $u(x, -\theta) = \mathcal{A}(\psi_h)(x,-\theta)$ is known.  

By Proposition \ref{SmallRemainder}, 
\[
u(x, -\theta) = J\psi_h(x, -\theta) + KJ\psi_h(x, -\theta) + R(x, -\theta),
\]
where $R(x, -\theta) = O(h^{2-n})$.  Moreover
\[
J\psi_h(x, -\theta) = \alpha(x, x_{\theta+})\psi_h(x_{\theta+}, -\theta) = 0
\]
since $\psi_h$ is only supported in a small neighbourhood of $(x, \theta)$.  Therefore
\[
u(x, -\theta) = KJ\psi_h(x, -\theta) + O(h^{2-n}).
\]

Now 
\begin{equation}\label{KJPsi}
KJ\psi_h(x, -\theta) = \int_0^{\|x - x_{\theta+}\|} \alpha(x, x + t\theta)\sigma(x + t \theta)\langle J\psi_h \rangle(x + t\theta)dt. 
\end{equation}
Along the line $\{x + t\theta\}$, 
\[
J\psi_h(x + t\theta,\theta) = \alpha(x + t\theta, (x + t\theta)_{\theta-})\psi_h((x + t\theta)_{\theta-}, \theta)
\]
Recall that $\psi_h(x,\theta) = X_h(x)\Theta_h(\theta)$. Then \eqref{SeparatedApproximations} implies that 
\[
\langle J\psi_h  \rangle(x + t\theta) = \alpha(x + t\theta, (x + t\theta)_{\theta-})X_h((x + t\theta)_{\theta-})+ o(h^{1-n}).
\]
Now $(x + t\theta)_{\theta-} = x$ and without loss of generality we can choose $X_h$ to be normalized so $X_h(x) = h^{1-n}$. Therefore we can write this as
\[
\langle J\psi_h \rangle(x + t\theta) = \alpha(x + t\theta, x)h^{1-n}+o(h^{1-n}).
\]
Substituting into \eqref{KJPsi}, we get
\[
KJ\psi_h(x, -\theta) = h^{1-n}\int_0^{\|x - x_{\theta+}\|} \alpha^2(x, x + t\theta)\sigma(x + t \theta)dt + o(h^{1-n}).
\]
Multiplying by $h^{n-1}$, we have
\[
h^{n-1}u(x, -\theta) = \int_0^{\|x - x_{\theta+}\|} \alpha^2(x, x + t\theta)\sigma(x + t \theta)dt + o_h(1).
\]
Now notice that \eqref{OpticalDistanceDerivative} implies that
\[
\alpha^2(x, x + t\theta)\sigma(x + t \theta) = -\half \partial_t \alpha^2(x, x + t\theta),
\]
so
\[
2h^{n-1}u(x, -\theta) = \alpha^2(x,x) - \alpha^2(x, x + \|x - x_{\theta+}\|\theta) + o_h(1).
\]
Since $\alpha(x, x) = 1$ and $x + \|x - x_{\theta+}\|\theta = x_{\theta+}$, we have
\begin{equation}\label{BoundaryInterpretation}
2h^{n-1}u(x, -\theta) = 1-\alpha^2(x, x_{\theta+}) + o_h(1)
\end{equation}
Taking the limit as $h \rightarrow 0$, we can see (no pun intended) that the quantity $\alpha^2(x, x_{\theta+})$ is determined by $\mathcal{A}|_E$, for any $(x, \theta) \in E_-$.
Now
\begin{equation}\label{LineIntegral}
\alpha^2(x, x_{\theta+}) = \exp\left(- 2\int_0^{\|x-x_{\theta+}\|}\sigma(x + s\theta) \, ds\right)
\end{equation}
so taking logarithms and dividing by $-2$ gives us the integral of $\sigma$ over the line from $x$ to $x_{\theta+}$. By varying $(x, \theta)$ within $E_-$, we recover the integral of $\sigma$ over every line segment in $\bar{\Om}$ which intersects $E$. The scannability condition on $E$ now guarantees that this information determines $\sigma$.  

\end{proof}

To prove Corollary \ref{ConvexCor}, note first that if $\Om$ is strictly convex, we can rewrite Definition \ref{Scannability} to say that $\Om$ can be scanned from $E$ if for $\ph \in C(\Om)$,
\[
T(\ph)(\ell) = 0 \mbox{ for each line } \ell \mbox{ intersecting } E \mbox{ implies that } \ph \equiv 0.
\]
(Here $T$ is the usual X-ray transform on $\R^n$.) Now we can make use of the following proposition, which follows directly from Theorem 5.1 in \cite{HamSmiSolWag} (see also similar theorems in \cite{Glo} and \cite{BomQui}).

\begin{prop}\label{XRayInversion}
Suppose $\ph \in L^1(\Om)$ and $V$ is an open set disjoint from the closure of the convex hull of $\Om$.  If $T(\ph)(\ell) = 0$ for all lines $\ell$ intersecting $V$ then $\ph \equiv 0$.
\end{prop}



\begin{proof}[Proof of Corollary \ref{ConvexCor}]

Given Proposition \ref{XRayInversion}, it suffices to find an open set $V$ outside $\bar{\Om}$ such that all lines through $V$ which intersect $\Om$ intersect $E$.

To do this, pick a point $p \in E$ and choose coordinates $(x_1, \ldots x_{n-1},y)$ on $\R^n$ so that $p$ lies at the origin, the tangent plane to $\Om$ at $p$ has equation $y = 0$, and $\Om$ lies in the upper half space $\{y > 0\}$. Locally around the origin, $\partial \Om$ coincides with the graph $\{y = f(x)\}$ of a nonnegative strictly convex function $f$, with $\Om \subset \{y \geq f\}$.

Since $f$ is strictly convex, for any compact set $K$ in the lower half space $\{ y < 0\}$, there exists a compact set $A \subset S^{n-1}$ excluding directions parallel to the plane $\{y=0\}$, such that only lines in directions inside $A$ can reach $\bar{\Om}$ from $K$.

Now let $B(0,r) \subset \R^n$ represent the ball of radius $r$ around the origin. The implicit function theorem guarantees the existence of a continuous function $g$ defined on $A \times B(0,r)$, for some $r$, such that $g(\theta, q)$ represents the point of intersection of the the graph of $f$ and the line through $q$ in direction $\theta$. Since $A$ is compact and $g(\theta,0) = 0$, it follows that for any $\e > 0$ there exists a $\delta > 0$ such that all lines through the ball $B(0,\delta)$ in directions in $A$ must pass through the graph of $f$ within a distance $\e$ of the origin. If $\e$ is small enough, this guarantees that all lines through $B(0,\delta)$ in directions contained in $A$ which intersect $\bar{\Om}$ intersect $E$.  Taking $V$ to be the interior of a ball compactly contained in $B(0,\delta) \cap \{y < 0\}$ finishes the proof. 

\end{proof}

\section{On Conjecture \ref{LocalConjecture}}

The final section of this paper is devoted to discussion of Conjecture \ref{LocalConjecture}. The major problem in proving the conjecture is that most approaches to the transport problem concentrate on highly singular ballistic parts to solutions of the RTE, which are supported along straight lines. The method described in the present paper is only a slight improvement, since it also relies on portions of the solution concentrated along these lines.  However, if there is a part of $\Om$ which cannot be reached from $E$ in a straight line through $\Om$, then a positive answer to this conjecture would require one to use the truly non-ballistic parts of solutions to ``see'' the hidden parts of $\Om$ from $E$.  

Some indication that this is possible might be inferred from the recent result of \cite{LaiLiUhl} (see also \cite{ChuLaiLi}), which analyzes the stability of the standard optical tomography problem of \cite{ChoSte} in the limit as $\sigma$ becomes large.  In this case the ballistic parts of the solutions decay exponentially and the stability of the problem goes to zero. Yet in the limit of large $\sigma$, one expects the solutions to the RTE to converge to solutions to a diffusion equation (see \cite{BenLioPap} and similar papers), and consequently one expects the optical tomography problem described above to turn into the Calder\'on problem (see \cite{LiNewStu} for one sense in which this holds). As the Calder\'on problem is solvable \cite{SylUhl}, this seems to imply that the non-ballistic terms of the solution to the RTE do convey useable information. Of course, the \emph{local} version of the Calder\'on problem remains open in general \cite{KenSal}, so not much is guaranteed. 

Note that the claim that the non-ballistic terms of the solution to the RTE convey usable information implies that seeing may not be restricted to straight lines. Indeed, if Conjecture \ref{LocalConjecture} holds, then it would imply that we can \emph{see around corners} in the presence of scattering (!). This is difficult to imagine but the above discussion seems to indicate that it might be possible.  In light of this interpretation (no pun intended), it would seem that Conjecture \ref{LocalConjecture} is worth some effort.

\end{document}